\documentclass[12pt]{amsart}
\usepackage{amsmath,amsfonts,amssymb,amsthm,amstext,pgf,graphicx,hyperref,verbatim,lmodern,textcomp,color,young,tikz}
\usetikzlibrary{decorations}
\usepackage[mathscr]{euscript}
\usetikzlibrary{decorations.markings}
\usetikzlibrary{arrows}
\usepackage{float}
\theoremstyle{plain}
\newtheorem{theorem}{Theorem}[section]
\newtheorem{lemma}[theorem]{Lemma}
\newtheorem{proposition}[theorem]{Proposition}
\newtheorem{corollary}[theorem]{Corollary}
\newtheorem{rem}[theorem]{Remark}

\theoremstyle{definition}

\newcommand{\Dom}{\mathrm{Dom}}
\newcommand{\cent}{\mathrm{cent}}

\title{}
\begin{document}
\title[Centralizer in finite groups and Domination number of their commuting graphs]{Centralizers in finite groups and Domination number of their commuting graphs} 
\author[Sudip Bera]{Sudip Bera}
\author[Hiranya Kishore Dey]{Hiranya Kishore Dey}
\author[Umang Jethva]{Umang Jethva}
\address[Sudip Bera]{Faculty of Mathematics, Dhirubhai Ambani University, Gandhinagar, India.}
\email{sudip\_bera@dau.ac.in}
\address[Hiranya Kishore Dey]{Department of Mathematics, Indian  Institute of Technology, Jammu, India.}
\email{hiranya.dey@iitjammu.ac.in} 
\address[Umang Jethva]{Dhirubhai Ambani University, Gandhinagar, India.}
\email{202421027@dau.ac.in} 
\maketitle	

\begin{abstract}
The proper commuting graph $\mathcal{C}^{**}(G)$ of a finite group $G$ is the simple graph whose vertices
are the noncentral elements of $G$ and two distinct vertices are adjacent if they commute. 
In this paper, we study the domination number and total domination number of proper commuting graphs of finite groups.
We first obtain general bounds for the domination number of proper commuting graphs.
For finite nilpotent groups, we exploit a strong product decomposition of commuting graphs to derive exact formulas for the domination number. We further determine
the exact domination number and total domination number for proper commuting graphs of several well-known families of finite groups, connecting with the centralizers of those groups. 
\end{abstract}
\section{Introduction}

The study of graphs associated with various algebraic structures has been a topic 
of active investigation during the last two decades. Such studies provide important insights into the interplay between algebraic structures and 
graph theoretic properties.
Several graphs have been introduced to explore algebraic structures using graph theory, including commuting graphs
of groups, power graphs of groups and semigroups \cite{surveypwrgraphkac1,undpwrgraphofsemgmainsgc1,compropofsemgrpkq4}, the 
centralizer graph \cite{lewis2026coveringingroups}, enhanced power graphs of groups \cite{firstenhcedpwrstrctreaacns1,enhancedpwrgrapbb3,BeraDey2022epow,bera-dey-sajal}, among others.

In this paper, we focus on the \emph{proper commuting graph of a group}. Let $G$ be a finite group, then the \emph{commuting graph} of $G$, denoted by $\mathcal{C}(G)$, is the simple graph whose vertex set is $G$, where two distinct vertices $x$ and $y$ are adjacent if and only if $xy = yx$. For a group $G,$ we denote by $e$ the identity element of $G$ and by $G^{**}$ the set $G\setminus Z(G),$ where $Z(G)$ is the center of the group $G.$ The induced subgraph of $\mathcal{C}(G)$ with the vertex set $G^{**}$ is called the {\it proper commuting graph} of $G$ and we shall denote it by $\mathcal{C}^{**}(G).$ 

Commuting graphs and 
their variants have been widely investigated due to their strong connections with group
structure \cite{ijraeljofmathematics, braurflower, onboundingdiamcommuting,haji2019groups}.
In particular, graph-theoretic parameters 
such as connectivity, diameter, 
and domination number have received considerable attention \cite{Araujo2011Minimal, giudici, IranmaneshJafarzadeh2008, MaCameron2024,SegevSeitz2002}. 

Haji and Amiri \cite{haji2019groups} studied the domination number of proper commuting graphs 
and established several notable results. Building on their work, we further examine domination 
in proper commuting graphs,
with an emphasis on how the algebraic structure of a group governs these properties.

\subsection{Basic Definitions and Notations}

Let $\Gamma$ be a finite simple graph with vertex set $V$ and edge set $E.$
For a vertex $v \in V$, 
the \emph{neighbourhood} of $v$ is the set $N(v)$ consisting of all vertices of $\Gamma$ that are adjacent to $v$. The \emph{closed neighbourhood} of $v$ is the set $N[v] = N(v) \cup \{v\}$.
A vertex $v \in V$ is
called a \emph{dominating vertex} of $\Gamma$ if its closed neighbourhood is $V$, that is, $N[v] = V$. A subset $X \subseteq V$ is called a \emph{dominating set} of $\Gamma$ if for every vertex $v \in V \setminus X$, there exists a vertex $u \in X$ such that $v$ is adjacent to $u$. Clearly, $V$ itself 
is a dominating set of $\Gamma$. 
If $\Gamma$ has no edges, then $V$ is the only dominating set of $\Gamma$.
A dominating set $X$ of $\Gamma$ is called a \emph{minimum dominating set} if $|X| \leq |Y|$ for every dominating set $Y$ of $\Gamma$. The cardinality of a minimum dominating set of $\Gamma$ is called the \emph{domination number} of $\Gamma$, and it is denoted by $\gamma(\Gamma)$. 
For a graph $\Gamma$, let $\Dom(\Gamma)$ denote the set of all universal
(dominating) vertices of $\Gamma$. Observe that $\Dom(\Gamma)\neq \emptyset$ if and only if $\gamma(\Gamma)=1$.

A set $T \subseteq V$ is called a \emph{total dominating set} if for every 
$v \in V$ there exists $u \in T$ such that $uv \in E$. 
The minimum cardinality of such a set is called the 
\emph{total domination number} of $\Gamma$ and is denoted by 
$\gamma_t(\Gamma)$. In this paper, our goal is to contribute to the study of the domination number and the total domination number of the proper commuting graph of group.

 Let $\Gamma_1=(V(\Gamma_1),E(\Gamma_1))$ and $\Gamma_2=(V(\Gamma_2),E(\Gamma_2))$ be simple graphs.
The \emph{strong product} of $\Gamma_1$ and $\Gamma_2$, denoted by $\Gamma_1 \otimes \Gamma_2$, is the graph
with the vertex set
\[
V(\Gamma_1 \otimes \Gamma_2)=V(\Gamma_1)\times V(\Gamma_2),
\]
where two distinct vertices $(g_1,h_1)$ and $(g_2,h_2)$ are adjacent in
$\Gamma_1 \otimes \Gamma_2$ if and only if at least one of the following holds:
\begin{enumerate}
  \item $g_1=g_2$ and $h_1 \sim h_2$ in $\Gamma_2$;
  \item $h_1=h_2$ and $g_1 \sim g_2$ in $\Gamma_1$;
  \item  $g_1 \sim g_2$ in $\Gamma_1$ and $h_1 \sim h_2$ in $\Gamma_2$.
\end{enumerate}


Recall that the centralizer of $x$ in $G,$ denoted by $C(x),$ is the subgroup of $G$ 
consisting of all elements that commute with $x;$ i.e., $C(x)=\{y\in G: xy=yx\}.$ 
By $\cent(G)$, we denote the set of all distinct centralizers of single elements of $G$. 
Researchers have studied the classification of finite groups by counting their centralizers;
see, for example, \cite{abdollahi2007groups,amiri2017groups,belcastro1994countingcentralizeringroups}.

For a finite group $G$, a covering of $G$ is a set $\Gamma= \{H_1, \ldots, H_n \} $ of proper subgroups $H_i$ of $G$ such that $G= \cup_{i=1}^{n} H_i$. 
A covering $\Gamma$ of a group $G$ is called centralizer covering if its members are 
centralizers of some elements in $G$. Following the notation of \cite{haji2019groups}, 
let $\sigma_c(G)$  denote the minimum possible cardinality of a centralizer covering of $G$. Clearly, $\sigma_c(G)= \gamma(\mathcal{C}^{**}(G)).$

\subsection{Main results}

Haji and Amiri in \cite[Theorem 4.1]{haji2019groups} gave an upper bound for $\gamma(\mathcal{C}^{**}(G))$ of every non-abelian group $G$.

\begin{theorem}
\label{thm:haji}
For any non-abelian group $G$, we have
\[
\gamma(\mathcal{C}^{**}(G)) \leq \displaystyle \frac{|G|-|Z(G)|+t}{2} 
\]
where $t$ is the number of centralizers of order $2$ in $G$. 
Moreover, equality holds if and only if $G$ is isomorphic to one of $S_3, D_8$ or $Q_8$.
\end{theorem}

In the following result, we give lower and upper
bounds for the domination number of the proper commuting graph of any finite non-abelian group.

\begin{theorem}\label{MainThm-1(UpperBndAndLwrBnd)}
Let $G$ be a finite non-abelian group. Then
\[
  \frac{|G|-|Z(G)|}{M} 
\leq \gamma\bigl(\mathcal{C}^{**}(G)\bigr)
\leq  T-U 
\]
where
\[
M = \max_{x \in G \setminus Z(G)} \left(|C(x)|\right) - |Z(G)|
\]
and $T$ denotes the number of maximal cyclic subgroups of $G$ and $U$ denotes the number of maximal cyclic subgroups of $G$, completely contained in $Z(G)$.
Moreover, if $p$ is the least prime divisor of $|G|$, then 
$\gamma(\mathcal{C}^{**}(G)) \geq p+1$. 
\end{theorem}

The inequalities in Theorem \ref{MainThm-1(UpperBndAndLwrBnd)} are tight in the sense that there exists groups (namely Heisenberg group of order $p^3$) for which both the inequalities are equalities. 
For groups which are not isomorphic to the generalized dihedral group, we also show the following upper bound which is similar to the one in Theorem \ref{thm:haji}. 

\begin{theorem}
\label{prop:haji-improved}
Let $G$ be any finite group not isomorphic to generalized dihedral group. Then, 
\[
\gamma(\mathcal{C}^{**}(G)) \leq  (|G|-|Z(G)|) \cdot \displaystyle \min \left\{  \frac{(1+\ln (d+1))}{d+1}, \frac{1}{2} \right\}
\] 
where $d= \min{|C(g)|-|Z(G)|}$. 
\end{theorem}


A common feature of many graphs, arising from groups is the presence of dominating vertices, which often correspond to distinguished elements of the group. To better understand the intrinsic structure of these graphs beyond such trivial influences, researchers have considered the induced subgraphs obtained by removing the dominating vertices. These reduced graphs frequently reveal more subtle combinatorial and algebraic properties. 

Motivated by this perspective, for a graph $\Gamma$ with a dominating vertex $v$, we define the induced subgraph $\Gamma \setminus \Dom(\Gamma)$ as the {\sl proper graph} of $\Gamma$. In this paper, we initiate a systematic study of proper graphs and investigate how they behave under graph operations, with particular emphasis on the strong product. Our aim is to understand how the removal of dominating vertices interacts with product structures and to identify structural properties preserved under this operation. Towards this, we prove the following.

\begin{theorem}
\label{thm:dom-strong-t}
Let $\Gamma_i$ be non-complete graphs for $1 \leq i \leq t$. 
If $\Dom(\Gamma_i)\neq\emptyset$ for all $i=1,\dots,t$. Then
\[
\gamma\!\left(
\Gamma_1\otimes\cdots\otimes\Gamma_t
\setminus
\bigl(
\Dom(\Gamma_1)\times\cdots\times\Dom(\Gamma_t)
\bigr)
\right)
=
\min_{1\le i\le t}
\gamma(\Gamma_i\setminus \Dom(\Gamma_i)).
\]
\end{theorem} 

Using Theorem \ref{thm:dom-strong-t}, we get an alternate proof of \cite[Corollary 3.6]{haji2019groups} which tells regarding the connection between the proper
commuting graph of a finite nilpotent group and the domination number of proper commuting graph of its Sylow subgroups. 
Subsequently, we turn our attention to the total domination number of the proper commuting graph of a finite group. Regarding the total domination number, we have the following theorem, which is proved in Section \ref{sec:total-dom-num}. 

\begin{theorem}\label{Thm:TotalDomOfProperCommutinGraph}
Let $G$ be a finite non-abelian group. Then the proper commuting graph $\mathcal{C}^{**}(G)$ admits a total dominating set if and only if $G$ is not isomorphic to a generalized dihedral group of order $2m,$ where $m$ is odd. 
Moreover, in this case,
\[
\gamma_t\bigl(\mathcal{C}^{**}(G)\bigr) \leq 2( T-U) 
\]
where $T,U$ are defined as in Theorem \ref{MainThm-1(UpperBndAndLwrBnd)}. 
\end{theorem} 

In the next result, we compare the total domination number of the proper graph of the strong product of graphs with the domination number of the proper graph of the individual graphs. 

\begin{theorem}
\label{thm:tot-dom-strong-t}
Let $\Gamma_1,\dots,\Gamma_t$ ($t \geq 2)$ be finite simple non-complete graphs such that
$\Dom(\Gamma_i)\neq\emptyset$ for all $i=1,\dots,t$.
Then
\[
\gamma_t\!\left(
\Gamma_1\otimes\cdots\otimes\Gamma_t
\setminus
\bigl(
\Dom(\Gamma_1)\times\cdots\times\Dom(\Gamma_t)
\bigr)
\right)
\le
\min_{1\le i\le t}
\gamma(\Gamma_i\setminus \Dom(\Gamma_i))+1.
\]
Moreover, if $\gamma_t(\Gamma_i \setminus \Dom(\Gamma_i)) > \gamma(\Gamma_i \setminus \Dom(\Gamma_i)) $ for each $i$, then 
\[
\gamma_t\!\left(
\Gamma_1\otimes\cdots\otimes\Gamma_t
\setminus
\bigl(
\Dom(\Gamma_1)\times\cdots\times\Dom(\Gamma_t)
\bigr)
\right)
=
\min_{1\le i\le t}
\gamma(\Gamma_i\setminus \Dom(\Gamma_i))+1.
\]
\end{theorem} 
Using Theorem \ref{thm:tot-dom-strong-t}, we prove Corollary \ref{Thm:TD-Nilpotent-Commuting} in Section \ref{sec:total-dom-num}.
In Section \ref{sec:applications}, we compute the exact domination number and total domination number of some well-known groups. In Section \ref{sec:spectrum-domination}, we consider the ratio of $\gamma(\mathcal{C}^{**}(G))$ and $|G|$ and prove results regarding its maximum value.

\section{Proof of Theorem \ref{MainThm-1(UpperBndAndLwrBnd)}}\label{Pf-MainThm1Bothbdry}

In this section, we start with the following connection between the domination number of $\mathcal{C}^{**}(G)$ and the maximal cyclic subgroups of $G$.

\begin{proposition}
\label{Thm:DomofCommtingGraphAndMaximalCyclicSubgrp}
Let $G$ be a finite non-abelian group. Then
\[
\gamma(\mathcal{C}^{**}(G)) \leq T - U,
\]
where $T$ denotes the number of maximal cyclic subgroups of $G$ and $U$ denote the number of maximal cyclic subgroups of $G$ contained in $Z(G)$.
\end{proposition}

\begin{proof}
Let 
\[
M_1 = \langle x_1 \rangle,\, M_2 = \langle x_2 \rangle,\, \dots,\, M_s = \langle x_s \rangle
\]
be all maximal cyclic subgroups of $G$ and without loss of generality, let $M_1, \dots, M_r$ 
be all maximal cyclic subgroups of $G$ which are contained in $Z(G)$. Consider the set 
\[
D = \{x_{r+1}, x_{r+2}, \dots, x_{s}\}.
\]
Clearly, $D \subseteq V(\mathcal{C}^{**}(G))$. We claim that $D$ is a dominating set of the graph $\mathcal{C}^{**}(G)$.
Let $x \in G^{**} \setminus D$. Then the cyclic subgroup $\langle x \rangle$ is contained in one of the $M_i$s. Moreover, $x \notin M_i$ for $i\in [r]$, otherwise $x \in Z(G)$, contradiction.  Hence, $x x_i = x_i x$ for some $i > r$. Therefore, $D$ is a dominating set of $\mathcal{C}^{**}(G)$.
\end{proof} 


\begin{proposition}
\label{Thm:lowe-bound-dom-any-grp}
Let $G$ be any finite group. Then 
\[
\gamma(\mathcal{C}^{**}(G)) \geq \frac{|G| - |Z(G)|}{M} 
\]
where \[
M = \max_{x \in G \setminus Z(G)} \bigl( |C(x)| \bigr) - |Z(G)| .
\]
In particular, if $p$ is the least prime divisor of $|G|$, then 
$\gamma(\mathcal{C}^{**}(G)) \geq p+1$. 
\end{proposition}

\begin{proof}
Let $\Gamma=\mathcal{C}^{**}(G) $. For $x \in G \setminus Z(G)$, the closed neighborhood in $\Gamma$ is
$
N_\Gamma[x] = C(x) \setminus Z(G),
$
with size
$
|N_\Gamma[x]| = |C(x)| - |Z(G)| \le M.
$
If $D$ is a dominating set in $\Gamma$, then
$
\bigcup_{d \in D} N_\Gamma[d] = G \setminus Z(G),
$
so that
$ \sum_{d \in D} |N_\Gamma[d]| \le |D| \cdot M.
$
Therefore,
\[
|D| \ge \frac{|G| - |Z(G)|}{M}.
\]
This completes the proof of the first part of the theorem. 

 We choose $x\notin Z(G)$ such that
 $
M=|C(x)|-|Z(G)|.
$ Let $k= |G|/|C(x)|$. 
Since $x$ is non-central, $C(x)\neq G$, hence $k>1$.  
By Lagrange's theorem, $k$ divides $|G|$.  
Because $p$ is the least prime divisor of $|G|$, it follows that
$
k\ge p.
$
As, $k=\frac{|G|}{|C(x)|} \geq p$, we must have 
$\frac{|G-Z(G)|}{|C(x)-Z(G)|} > p$, completing the proof. 
\end{proof}

Now, Theorem \ref{MainThm-1(UpperBndAndLwrBnd)} follows from Propositions \ref{Thm:DomofCommtingGraphAndMaximalCyclicSubgrp} and  \ref{Thm:lowe-bound-dom-any-grp}.

\begin{corollary}\label{Lemma:DomCommAnygrp-ge-3}
Let $G$ be any finite non-abelian group. Then 
\[
\gamma(\mathcal{C}^{**}(G)) \ge 3.
\]
\end{corollary} 




\section{Proof of Theorem \ref{prop:haji-improved}}

In this section, we establish bounds for the domination number of the proper commuting graph of general non-abelian groups. To this end, we first characterize all finite groups $G$ that admit a non-central element whose centralizer has order $2$. For this purpose, we recall the notion of a generalized dihedral group.

For a finite abelian group $A$, the \emph{generalized dihedral group} $D(A)$ is defined as the semidirect product of $A$ with a cyclic group $\langle x \rangle$ of order $2$, where
\[
x^{-1} a x = a^{-1} \quad \text{for all } a \in A.
\]
(When $A$ is cyclic, this reduces to the usual dihedral group.)
It is straightforward to verify that every element of $D(A) \setminus A$ has order $2$. Moreover, if $x' \in D(A) \setminus A$, then the centralizer of $x'$ in $A$ is precisely the set of elements of $A$ of order $1$ or $2$.

\begin{theorem}
\label{thm:dom-gen-dih}
If $A$ is not elementary abelian, the domination number of $\mathcal{C}^{**}(D(A))$ is 
\[
1+\frac{|A|}{|T|},
\]
where
$
T=\{a\in A:a^2=1\}.
$ 
\end{theorem} 

\begin{proof}

The noncentral elements of $D(A)$ consist of two types of elements:
\[
A\setminus T
\qquad\text{and}\qquad
AT=\{at:a\in A, t \in T\}.
\]
Since \(A\) is abelian, every two elements of \(A\) commute. Hence
$A \setminus T$
induces a complete subgraph.
Therefore this component has domination number $1$.

Take $at,bt' \in AT$. Then
$
(at)(bt')=ab^{-1}tt',
$
while
$
(bt')(at)=ba^{-1}t't.
$

We have
\[
(at)(bt)=(bt)(at)
\iff
ab^{-1}=ba^{-1}
\iff
a^2=b^2
\iff
ab^{-1}\in T.
\]

Thus
$
at$ commutes with $bt'$
if and only if 
$aT=bT.$
Hence $AT$ splits into exactly
$
[A:T]=\frac{|A|}{|T|}
$
pairwise disjoint complete subgraphs, indexed by the cosets of \(T\) in \(A\),
and 
each such clique has domination number $1$. 

Moreover, if $a\in A\setminus T$ and 
$bt\in At$, then
$
a(bt)=abt,
$
while
$
(bt)a=ba^{-1}t.
$
These are equal if and only if
$
a=a^{-1},
$
equivalently,
$
a^2=1.
$ 
But as $a\notin T$, so $a$ and $b$ do not commute.
Hence there are no edges between $A \setminus T$ and $BT$. 

Therefore the proper commuting graph is a disjoint union of
\[
1+\frac{|A|}{|T|}
\]
complete graphs, completing the proof. 
\end{proof}

\begin{theorem}
\label{CentralizerIs2-GeneralizedDihedralGrp}
 Let $G$ be a finite group. Suppose there exists an element $x \in G \setminus Z(G)$ such that $|C(x)| = 2$. Then $G$ is a generalized dihedral group of order $2m,$ where $m$ is odd.   
\end{theorem}

\begin{proof}
First, we prove that the order of $G$ must be twice an odd number. Let $P$ be a Sylow $2$-subgroup of $G$. Suppose that $|P| > 2$, and let $x \in P$.

Since $P$ is a finite $2$-group, its center $Z(P)$ is nontrivial, that is,
\[
Z(P) \neq \{e\}.
\]
Hence, there exists an element $z \in Z(P)$ with $z \neq e$.

Now consider the element $x \in P$. If $x \in Z(P)$, then $x$ lies in the center of $P$. If $x \notin Z(P)$, then since $z \in Z(P)$, we have
\[
zx = xz,
\]
so $x$ commutes with a nontrivial element $z \in Z(P)$.
Thus, in either case, either $x \in Z(P)$ or $x$ commutes with some nontrivial element of $P$, both of which contradict the hypothesis. Therefore, $|P| = 2$.
Since all Sylow $2$-subgroups of $G$ are conjugate, it follows that every Sylow $2$-subgroup of $G$ has order $2$. Hence, the order of $G$ is twice an odd number.
Let $G$ be a finite group of order $2m$, where $m$ is odd. Then $G$ has a normal subgroup $N$ of index $2$.
Now, $x \in G \setminus N,$ and $x$ normalizes $N$, and hence acts on $N$ by conjugation. We show that, this action fixes only the identity element of $N.$ In fact, \[
xnx^{-1} = n \implies xn=nx \implies n\in C(x).
\]
This is a contradiction, as $N$ is odd order subgroup and the order of $x$ is $2.$ Here we prove that $xnx^{-1} = n^{-1} \quad \text{for all } n \in N.$
Define $\varphi : N \to N$ by
\[
\varphi(n) = xnx^{-1}.
\]
Then $\varphi$ is an automorphism of $N$. Moreover, since $x^2 = e$, we have
\[
\varphi^2(n) = x(xnx^{-1})x^{-1} = n,
\]
so $\varphi$ is an involution. Thus, for each $n \in N$, the element $\varphi(n)$ is such that applying $\varphi$ twice returns $n$. 
Since the only fixed point of $\varphi$ is $e$, no nontrivial element is mapped to itself. Therefore, for every $n \neq e$, we must have $\varphi(n) \neq n.$ Now, in a finite group of odd order, 
the only possibility for such an involutory automorphism without nontrivial fixed points is inversion. Hence,
\[
\varphi(n) = n^{-1} \quad \text{for all } n \in N.
\]

Therefore,
\[
xnx^{-1} = n^{-1} \quad \text{for all } n \in N.
\]

We now show that $N$ is abelian. For any $a,b \in N$, we have
\[
x(ab)x^{-1} = (ab)^{-1} = b^{-1}a^{-1},
\]
while
\[
xax^{-1}\, xbx^{-1} = a^{-1}b^{-1}.
\]
Comparing these expressions yields $ab = ba$, and hence $N$ is abelian.

Thus, $G$ is generated by the abelian subgroup $N$ together with an element $x$ of order $2$ such that
\[
xnx^{-1} = n^{-1} \quad \text{for all } n \in N.
\]
Therefore,
\[
G \cong N \rtimes \langle x \rangle,
\]
where $x$ acts on $N$ by inversion. Consequently, $G$ is a generalized dihedral group, where $N$ is an abelian subgroup of odd order.
Hence $G$ is a generalized dihedral group over the abelian group $N$ of odd order.
\end{proof}

We can now prove the following upper bound for any nonabelian group $G$.

\begin{proof}
[Proof of Theorem \ref{prop:haji-improved}] 
For any element $g \in G \setminus Z(G)$, the centralizer $C(g)$ properly contains $Z(G)$. Moreover, as the order of $Z(G)$ divides $C(g)$ and by Theorem \ref{CentralizerIs2-GeneralizedDihedralGrp}, $|C(g)| \geq 3$ for any $g \in G \setminus Z(G)$, we have $|C(g)|-|Z(G)| \geq 2$ for any $g \in G \setminus Z(G)$. Using \cite{alon},
we have 
\[
\gamma(\mathcal{C}^{**}(G)) \leq  \displaystyle    \frac{ (|G|-|Z(G)|)(1+\ln (d+1))}{1+d}.  
\] 
As $Z(G)$ is nontrivial, $|C(g)|>2$. Therefore, using Theorem \ref{thm:haji}, we also have \[
\gamma(\mathcal{C}^{**}(G)) \leq  \displaystyle    \frac{ (|G|-|Z(G)|)}{2}.  
\] 
This completes the proof.  
\end{proof}

\begin{rem}
When $d$ is large, $\frac{1+\ln(d+1)}{1+d}$ is much smaller than $1/2$ and thus Proposition \ref{prop:haji-improved} gives a stronger bound than Theorem \ref{thm:haji} in that case. 
\end{rem}

\section{Proof of Theorem \ref{thm:dom-strong-t}}\label{PfMainThm2(UpperBnd-NilpotentGrp)}

In this section, we concentrate on graphs with dominating vertices. 
Let $\Gamma_1,\dots,\Gamma_t$ be graphs, and denote by
\[
\Gamma_1\otimes\cdots\otimes\Gamma_t
\]
their strong product. The following two propositions are immediate and so we omit the proofs. 

\begin{proposition}
For any graphs $\Gamma_1,\dots,\Gamma_t$,
\[
\Dom(\Gamma_1\otimes\cdots\otimes\Gamma_t)
=
\Dom(\Gamma_1)\times\cdots\times\Dom(\Gamma_t).
\]
\end{proposition}


\begin{proposition}
For any graphs $\Gamma_1,\dots,\Gamma_t$,
\[
\gamma(\Gamma_1\otimes\cdots\otimes\Gamma_t)
\le
\prod_{i=1}^t \gamma(\Gamma_i).
\]
\end{proposition}


\begin{proof}[Proof of Theorem \ref{thm:dom-strong-t}] 
Let
\[
\Gamma
=
\Gamma_1\otimes\cdots\otimes\Gamma_t
\setminus
\bigl(
\Dom(\Gamma_1)\times\cdots\times\Dom(\Gamma_t)
\bigr),
\]
and set
\[
k
=
\min_{1\le i\le t}
\gamma(\Gamma_i\setminus \Dom(\Gamma_i)).
\]

\medskip
\noindent
We at first show that $\gamma(\Gamma) \leq k$. Without loss of generality, assume
$
k=\gamma(\Gamma_1\setminus \Dom(\Gamma_1)).
$
Let $D_1$ be a minimum dominating set of
$\Gamma_1\setminus \Dom(\Gamma_1)$, and for each $j=2,\dots,t$ fix a vertex
$u_j\in \Dom(\Gamma_j)$.
Define
\[
D=\{(v,u_2,\dots,u_t): v\in D_1\}.
\]
We claim that $D$ dominates $\Gamma$.
Let $(x_1,\dots,x_t)\in V(\Gamma)$.
If $x_1\notin \Dom(\Gamma_1)$, then there exists $v\in D_1$
such that $v=x_1$ or $v$ is adjacent to $x_1$ in $\Gamma_1$.
Since each $u_j$ is universal in $\Gamma_j$ for $j\ge 2$,
the vertex $(v,u_2,\dots,u_t)$ is adjacent to $(x_1,\dots,x_t)$
in the strong product. If $x_1 \in \Dom(\Gamma_1) $, then $(x_1, x_2, \dots, x_t)$ is of course adjacent to $(v, u_2, \dots, u
_t)$. 
Thus $D$ dominates $\Gamma$, and
\[
\gamma(\Gamma)\le |D|=k.
\]

\medskip
\noindent
We next show that $k$ is also a lower bound. 
Suppose, for contradiction, that there exists a dominating set
$D\subseteq V(\Gamma)$ with $|D|<k$.
For each $i$, define the projection
\[
\pi_i(D)=\{x_i\in V(\Gamma_i):
(x_1,\dots,x_t)\in D \text{ for some } x_j\}.
\]
Since $|\pi_i(D)|\le |D|<\gamma(\Gamma_i\setminus \Dom(\Gamma_i))$,
the set $\pi_i(D)$ does not dominate
$\Gamma_i\setminus \Dom(\Gamma_i)$.
Hence there exists
\[
x_i^0\in V(\Gamma_i)\setminus \Dom(\Gamma_i)
\]
which is neither equal nor adjacent to any vertex of $\pi_i(D)$.

The vertex $(x_1^0,\dots,x_t^0)$ belongs to $V(\Gamma)$.
Let $(x_1,\dots,x_t)\in D$.
Then for each $i$, $x_i$ is neither equal nor adjacent to $x_i^0$ in $\Gamma_i$,
so $(x_1,\dots,x_t)$ is not adjacent to $(x_1^0,\dots,x_t^0)$
in the strong product.
Thus $(x_1^0,\dots,x_t^0)$ is not dominated by $D$,
a contradiction.

Hence $\gamma(\Gamma)\ge k$, completing the proof.
\end{proof}

In the next result, we show that the domination number of the proper graph of a graph does not change if we take the strong product with a complete graph.

\begin{lemma}\label{thm: Dom_St_Kn}
    Let $\Gamma = \Gamma_1 \otimes K_n$, where $\Gamma_1$ is any non-complete simple graph and $K_n$ denotes the complete graph with $n$ vertices. Let $\Dom(\Gamma_1) \neq \emptyset$. Then
    \[
    \gamma(\Gamma\setminus \Dom(\Gamma)) = \gamma(\Gamma_1 \setminus \Dom(\Gamma_1)).
    \]
\end{lemma}

\begin{proof}
    Let 
    $
    k= \gamma(\Gamma_1 \setminus \Dom(\Gamma_1))$
    and $D_1$ be a minimum dominating set of $\Gamma_1\setminus \Dom(\Gamma_1)$. Let $u \in V(K_n)$. 
    Define
    \[
    D=\{(v,u): v \in D_1\}.
    \]
    We claim that $D$ dominates $\Gamma\setminus \Dom(\Gamma)$. Let $(x,y) \in V(\Gamma\setminus \Dom(\Gamma))$. Then $(x,y) \notin \Dom(\Gamma)$, By the theorem, $\Dom(\Gamma) = \Dom(\Gamma_1) \times K_n$, and hence $x\notin \Dom(\Gamma_1)$. If $x \notin (\Dom(\Gamma_1))$, then there exists $v\in D_1$ such that $v=x$ or $v$ is adjacent to $x$ in $\Gamma_1$. Since $u \in V(K_n)$, the vertex $(v,u)$ is adjacent to $(x,y)$ in the strong product. 
    Thus $D$ dominates $\Gamma\setminus \Dom(\Gamma)$, and 
    \[
    \gamma(\Gamma\setminus \Dom(\Gamma)) \leq |D| =  \gamma(\Gamma_1\setminus \Dom(\Gamma_1)).
    \]
   
    Suppose, there exists a dominating set $D \subseteq V(\Gamma\setminus Dom(\Gamma))$ with $|D| < k$. Then 
    \[
    |D|< \gamma(\Gamma_1\setminus Dom(\Gamma_1))
    \]
    Consider the following projections
    \[
    \pi_1(D) = \{x\in V(\Gamma_1):(x,y)\in D \text{ for some } y\},
    \]
    and 
    \[
    \pi_2(D) = \{y\in V(K_n):(x,y)\in D \text{ for some } x\}.
    \]
    Since $|\pi_1(D)| < |D_1|$, the set $\pi_1(D)$ does not dominate $\Gamma_1\setminus \Dom(\Gamma_1)$. Hence there exists 
    \[
    x_0 \in V(\Gamma_1)\setminus \Dom(\Gamma_1)
    \]
    such that $x_0$ is neither equal nor adjacent to any vertex of $\pi_1(D)$.
    Take a vertex $(x_0,u_1) \in V(\Gamma\setminus \Dom(\Gamma))$, where $u_1\in V(K_n)$. 
    Let $(x,y)\in D$ be arbitrary. Then $x\in \pi_1(D)$ and $y\in \pi_2(D)$. By construction, $x$ is neither equal nor adjacent to $x_0$ in $\Gamma_1$. Therefore, $(x,y)$ is not adjacent to $(x_0,u_1)$ in the strong product. 
    Thus, $(x_0,u_1)$ is not dominated by $D$, contradicting the assumption that $D$ is a dominating set of $\Gamma\setminus \Dom(\Gamma)$.
    Hence $|D| \geq k$.
    This completes the proof.  
\end{proof}

The following proposition follows from the definition of the commuting graph.   

\begin{proposition}\label{Prop:Com_strong-product}
Let $G_1,G_2,\dots,G_n$ be finite groups. Then
\[
\mathcal C(G_1\times G_2\times \cdots \times G_n) = \mathcal C(G_1)\otimes \mathcal C(G_2)\otimes \cdots \otimes \mathcal C(G_n).
\]
\end{proposition}

For a finite group $G$, let $\pi(G)$ denote the set of prime divisors of $G$. 

\begin{theorem}
\label{MainThm-2(UpperBnd-NilpotentGrp)} 
Let $G$ be a finite nilpotent group with $\pi(G)=\{p_1,p_2,\ldots,p_t\}$. 
For each $i \in [t]$, let $P_i$ be a Sylow $p_i$-subgroup of $G$. 
Then
\[
\gamma\bigl(\mathcal{C}^{**}(G)\bigr)
=
\min_{\substack{1 \le i \le t \\ P_i \text{ non-abelian}}}
\gamma\bigl(\mathcal{C}^{**}(P_i)\bigr).
\]
\end{theorem}

\begin{proof}
Proposition~\ref{Prop:Com_strong-product} implies that
\[
\mathcal{C}(P_1 \times \cdots \times P_t)
=
\mathcal{C}(P_1) \otimes \cdots \otimes \mathcal{C}(P_t).
\]
Let $S= \{i \in [t]: P_i \mbox{ non-abelian}\}$. Then, 
\[
\mathcal{C}(P_1 \times \cdots \times P_t)
=  \mathcal{C}( \prod_{i \in S} P_i ) \otimes K_r
\]
where $r=\prod_{i \in [t] \setminus S} |P_i|$. 
Since each $P_i$ is non-abelian for $i \in S$, the commuting graph $\mathcal{C}(P_i)$ is non-complete for $i \in S$. Moreover, $\mathrm{Dom}(\mathcal{C}(P_i)) = Z(P_i) \neq \emptyset$
for each $ i \in S.$ Therefore, 
using Theorem~\ref{thm:dom-strong-t}, we obtain
\[
\gamma\bigl(\mathcal{C}^{**}( \prod_{i \in S} P_i)\bigr) 
=
\min_{i \in S}
\gamma\bigl(\mathcal{C}^{**}(P_i)\bigr).
\]
We now use Lemma \ref{thm: Dom_St_Kn} to get 
\[
\gamma\bigl(\mathcal{C}^{**}(G)\bigr)
= \gamma\bigl(\mathcal{C}^{**}( \prod_{i \in S} P_i)\bigr) = 
\min_{i \in S}
\gamma\bigl(\mathcal{C}^{**}(P_i)\bigr).
\]
This completes the proof.
\end{proof}


\begin{corollary}
Let $G$ be a finite nilpotent group such that each $P_i$ is non-abelian and $|P_i|=8$ for some $i\in [t]$. Then 
\[
\gamma(\mathcal{C}^{**}(G))=3.
\]
\end{corollary}

\begin{proof}
Since $|P_i|=8$, we have $|Z(P_i)|=2$. Hence, by Corollary~\ref{cor:heisenberg}, 
\[
\gamma(\mathcal{C}^{**}(P_i))=3.
\]
Moreover, by Corollary~\ref{Lemma:DomCommAnygrp-ge-3}, we have 
\[
\gamma(\mathcal{C}^{**}(P_i)) \ge 3
\]
for each $i$. Therefore, by Theorem 
\ref{MainThm-2(UpperBnd-NilpotentGrp)}, it follows that
\[
\gamma(\mathcal{C}^{**}(G))=3.
\]
\end{proof} 

As another application of Theorem \ref{thm:dom-strong-t}, we get the domination number of the proper enhanced power graph of a nilpotent group. We at first recall the definition of  enhanced power graph. Given a group $G$, the enhanced power graph of $G$, denoted by $\mathcal{G}_E(G)$, is the graph with vertex set $G$, in which two vertices $u$ and $v$ are adjacent if and only if there exists an element $w \in G$ such that both $u$ and $v$ are powers of $w$. The \textit{proper enhanced power graph} of $G$, denoted by $\mathcal{G}_E^{**}(G)$, is the graph obtained from the enhanced power graph $\mathcal{G}_E(G)$ by deleting all dominating vertices.

\begin{corollary}[Theorem 2.5, \cite{BeraDey2022epow}]
    Let $ G_1 $ be a product of non-cyclic $ p $-groups of the form:
    \[
    G_1 = P_1 \times P_2 \times \cdots \times P_m,
    \]
    where $ m \geq 2 $ and, for each $i\in [m]$, $ P_i $  is a $ p_i $-group which is neither cyclic nor generalized quaternion. Let $ s_i $ be the number of distinct $ p_i $-order subgroups of $ G_1 $.
    Then,
    \[
    \gamma(\mathcal G^{**}_E(G_1)) = \min_{i \in [m]} \{s_i\}. 
    \]
\end{corollary} 

\begin{proof}
Since \[G_1 = P_1 \times P_2 \times \cdots \times P_m,\] where $ m \geq 2 $ and each $ P_i $ is a $p_i$-group, where $ p_i $ is a prime number, it follows that $\gcd(|P_i|,|P_j|)=1$ for all $i\ne j$, where $i,j\in [m]$. It is easy to see that:
\[
\mathcal G_E(G_1) = \mathcal G_E(P_1) \otimes \mathcal G_E(P_2) \otimes \cdots \otimes \mathcal G_E(P_m).
\]
Applying Theorem~\ref{thm:dom-strong-t}, we obtain:
\[
\gamma(\mathcal G_E(G_1) \setminus \text{Dom}(\mathcal G_E(G_1))) = \min_{i \in [m]} \big\{ \gamma(\mathcal G_E(P_i) \setminus \text{Dom}(\mathcal G_E(P_i))) \big\},
\]
which leads to:
\[
\gamma(\mathcal G^{**}_E(G_1)) = \min_{i \in [m]} \big\{ \gamma(\mathcal G^{**}_E(P_i)) \big\}.
\] 
Let $G_1$ be a finite $p$-group which is neither cyclic nor generalized quaternion. By [Theorem 2.4, \cite{BeraDey2022epow}], $\gamma\!\left(\mathcal{G}^{**}_E(G_1)\right)$ is equal to the number of distinct $p$-order subgroups of $G_1$. This completes the proof.  
\end{proof}


 \section{Proof of Theorems \ref{Thm:TotalDomOfProperCommutinGraph} and \ref{thm:tot-dom-strong-t}}
 \label{sec:total-dom-num} 

In this section, we first present the proof of Theorem~\ref{Thm:TotalDomOfProperCommutinGraph}. 
 
\begin{proposition}\label{Prop:CommutingGraphHasTDIff}
Let $G$ be a finite non-abelian group. Then $\mathcal{C}^{**}(G)$ possesses a total dominating set if and only if 
\[
|C(x)| \geq 3, \quad \text{for all } x \in G \setminus Z(G).
\]
\end{proposition}

\begin{proof}
First suppose that $G$ is a non-abelian group such that $|C(x)| \geq 3$ for all $x \in G \setminus Z(G)$. Since $Z(G)$ is a subgroup of $C(x)$, by Lagrange's theorem we have $|Z(G)| \leq |C(x)|-2$. Therefore, each $C(x)$ contains at least two elements from $G \setminus Z(G)$.

Now consider a set $T$ obtained by taking two elements, namely $x$ and any element $y \in C(x)$, from each centralizer $C(x)$. Clearly, every vertex in $\mathcal{C}^{**}(G) \setminus T$ is adjacent to at least one vertex in $T$, and each vertex of $T$ is also adjacent to another vertex of $T$. Hence, $T$ forms a total dominating set of the graph $\mathcal{C}^{**}(G)$.

Conversely, suppose there exists $x \in G \setminus Z(G)$ such that $|C(x)| = 2$. Since the identity element $e$ of the group belongs to $C(x)$, it follows that $C(x) = \{e,x\}$. Thus $x$ has no adjacent vertex in $\mathcal{C}^{**}(G)$, and hence $x$ is an isolated vertex of the graph. Therefore, $\mathcal{C}^{**}(G)$ does not possess any total dominating set.
This completes the proof.
\end{proof}

\begin{proof}[Proof of Theorem \ref{Thm:TotalDomOfProperCommutinGraph}]
By Proposition \ref{Prop:CommutingGraphHasTDIff}, and Theorem \ref{CentralizerIs2-GeneralizedDihedralGrp}, $\mathcal{C}^{**}(G)$ admits a total dominating set if and only if $G$ is not isomorphic to a generalized dihedral group of order $2m,$ where $m$ is odd, completing the proof.
\end{proof}

We next give the proof of Theorem \ref{thm:tot-dom-strong-t} which compares the total domination number of the proper graph of the strong product of graphs with the total domination number of the proper graph of the individual graphs. 


\begin{proof}[Proof of Theorem \ref{thm:tot-dom-strong-t}]
Let
\[
\Gamma
=
\Gamma_1\otimes\cdots\otimes\Gamma_t
\setminus
\bigl(
\Dom(\Gamma_1)\times\cdots\times\Dom(\Gamma_t)
\bigr).
\]
We at first prove the first statement. Let $k=\min_{1\le i\le t}\gamma(\Gamma_i\setminus \Dom(\Gamma_i))$,
and assume without loss of generality that
$k=\gamma(\Gamma_1\setminus \Dom(\Gamma_1))$.
Let $D_1$ be a minimum dominating set of
$\Gamma_1\setminus \Dom(\Gamma_1)$.
For each $j=2,\dots,t$, fix $u_j\in \Dom(\Gamma_j)$ and define
\[
D'=\{(v,u_2,\dots,u_t): v\in D_1\}.
\]
Every vertex of $\Gamma\setminus D'$ is dominated by some vertex of $D'$.
Choose $u_1\in \Dom(\Gamma_1)$ and choose $w_j\notin \Dom(\Gamma_j)$
for some $j\ge 2$ (such a vertex exists by assumption).
Set
\[
D=D'\cup\{(u_1,w_2,\dots,w_t)\}.
\]
The vertex $(u_1,w_2,\dots,w_t)$ is adjacent to all vertices of $D'$,
so $D$ is a total dominating set of $\Gamma$.
Hence
\[
\gamma_t(\Gamma)\le k+1.
\]
This completes the proof of the first statement.

Now suppose, $\gamma_t(\Gamma_i \setminus \Dom(\Gamma_i)) > \gamma(\Gamma_i \setminus \Dom(\Gamma_i)) $ for each $i$ and we want to show that $\gamma_t(\Gamma)=k+1$. We will prove by contradiction. Towards that, we first prove the following claim: 

{\bf Claim:}
Let $S$ be a total dominating set of $\Gamma$. Define $S_j$ to be the projection of $S$ onto $j$-th co-ordinate, that is, 
\[ S_j = \{ v \in \Gamma_j: (v_1, v_2, \dots, v_{j-1}, v, v_{j+1}, \dots, v_t) \in S \mbox{ for some } v_i \} \]
Then there exists $j$ such that $S_j$ dominates $\Gamma_j \setminus \Dom(\Gamma_j)$.

{\bf Proof of the claim:}
Assume the contrary.
Then for every $j$ there exists $x_j\in \Gamma_j$
such that no vertex of $S$ is adjacent to $x_j$
in coordinate $j$.
Consider $v=(x_1,\dots,x_t)$.
For any $s=(s_1,\dots,s_t)\in S$,
there exists some coordinate $j$ such that
$s_j$ is not adjacent to $x_j$.
Hence $s$ is not adjacent to $v$, contradiction. This proves the claim. 

Let $S$ be a minimum total dominating set of $\Gamma$.
By the claim, there exists $j$ such that $S_j$ dominates $\Gamma_j \setminus \Dom(\Gamma_j)$.
Therefore
\[
|S|\ge \gamma(\Gamma_j \setminus \Dom(\Gamma_j))\ge k.
\]

If $|S|=k$, then its projection onto $\Gamma_j \setminus \Dom(\Gamma_j)$
would be a dominating set of size $k$.
But by assumption,
$\gamma_t(\Gamma_j \setminus (\Dom(\Gamma_j))>m$,
so no dominating set of size $m$
can be total in $\Gamma_j \setminus (\Dom(\Gamma_j)$.
Hence within $S$ some vertex would lack
a neighbour in $S$,
contradicting total domination in $\Gamma$.

This completes the proof of the second part of the theorem. 
\end{proof}

\begin{corollary}
\label{Thm:TD-Nilpotent-Commuting}
Let $G = P_1 \times \cdots \times P_t$ $(t \geq 2)$ be a finite nilpotent group such that none of the $P_i$ is abelian. Then
\[
\gamma_t\bigl(\mathcal{C}^{**}(G)\bigr)
\leq 
\min_{1 \leq i \leq t} \gamma\bigl(\mathcal{C}^{**}(P_i)\bigr) + 1.
\]
\end{corollary}

\begin{proof}
From the proof of Theorem \ref{MainThm-2(UpperBnd-NilpotentGrp)} we have
\begin{equation*}
\mathcal{C}^{**}(P_1 \times \cdots \times P_t)
= \mathcal{C}(P_1) \otimes \cdots \otimes \mathcal{C}(P_t)
\setminus \bigl(\mathrm{Dom}(\mathcal{C}(P_1)) \times \cdots \times \mathrm{Dom}(\mathcal{C}(P_t))\bigr).
\end{equation*} 
The proof now follows using Theorem \ref{thm:tot-dom-strong-t}. 
\end{proof}

Repeatedly using Theorem \ref{thm:tot-dom-strong-t} and Lemma \ref{thm: Dom_St_Kn}, alongwith \cite[Theorem 4.2]{bera-dey-sajal} and \cite[Theorem 5.1]{BeraDey2022epow}, we can get the results of \cite{Bera2024strongdomination} as a simple corollary.

\begin{corollary}
\cite[Theorems 3.2, 3.3, 3.4. 3.5]{Bera2024strongdomination}
Let G be any finite nilpotent group that does not contain a maximal subgroup of order $2$, and $G_1$ be a finite nilpotent group with no Sylow subgroups that are cyclic or generalized quaternion. Then,
    \[
\gamma_t(\mathcal G_E^{**}(G)) = 
\begin{cases}

2^{k-1}+2,  & \text{if } G = Q_{2^k} \text{ or } \\
            & G = \mathbb Z_n \times Q_{2^k},\ \gcd(n,2)=1 \\[6pt]

2s_1,         & \text{if } G=G_1=P_1 \text{ or } G=P_1 \times \mathbb Z_n,\ \gcd(p_1,n)=1\\[6pt]

s+1       & \text{if } G = G_1 \text{ or } G = G_1 \times \mathbb Z_n,\\ 
            &G_1= \prod_{i=1}^m P_i, (m\geq 2),\ \gcd(|G_1|,n)=1 \\[6pt]

\min\{s,2^{k-2}+1\}+1,  &\text{if }  G = G_1 \times Q_{2^k},\ \gcd(|G_1|,2)=1, \\
                        & \text{or } G = G_1 \times Q_{2^k} \times \mathbb{Z}_n, \\
                        & \gcd(|G_1|,2)=\gcd(n,2)=1.
\end{cases}
\]
where $ s_i $ is the number of distinct $ p_i $-order subgroups of $G_1$ and $s = \min\{ s_1,s_2,\dots, s_m\}$.
\end{corollary}

\section{Applications}
\label{sec:applications}

In this section, we compute the domination number and total domination number of the proper commuting graph of some well-known groups. For a group $G$, let $\mathrm{cent}(G)$ denote the set of all centralizers of single elements of $G$, that is, $\mathrm{cent}(G) = \{ C(g) \mid g \in G \}.$ 
We start with {\it AC-group}. Recall that $G$ is called an AC-group if $C(x)$ is abelian for any $x \in G \setminus Z(G)$. In \cite[Proposition 2.2]{haji2019groups}, Haji and Amiri showed the following. 

\begin{proposition}\label{prop:dom-AC-grp}
Let $G$ be any finite non-abelian group. Then, $G$ is an AC-group if and only if $\gamma(\mathcal{C}^{**}(G))=|\cent(G)|-1$. 
\end{proposition}

\begin{proposition}
\label{prop:total-dom-AC-group}
Let $G$ be a finite non-abelian group which is not generalized dihedral. Then, $\gamma_t(\mathcal{C}^{**}(G))= 2|\cent(G)|-2$ if and only if $G$ is an AC-group. 
\end{proposition}

\begin{proof}
Let $G$ be an AC-group. Then, for $a,b \in G \setminus Z(G)$, we have $ab=ba$ if and only if $C(a)=C(b)$. Let $x, y \in G \setminus Z(G)$. If $C(x) \neq C(y)$, then we claim that $C(x)  \cap C(y)= Z(G)$. Otherwise, if $w \in [C(x) \cap C(y)] \setminus Z(G)$ then $wx=xw$. So, $C(w)=C(x)$ and similarly, $C(w)=C(y)$. 
Therefore, centralizers of two non-central elements $x,y$ are different  if and only if $C(x) \setminus Z(G)$ and $C(y) \setminus Z(G)$ are distinct components of the proper commuting graph of $G$. Therefore, $\gamma_t(\mathcal{C}^{**}(G))= 2\gamma(\mathcal{C}^{**}(G))$, completing the proof. 

Conversely, if $\gamma_t(\mathcal{C}^{**}(G))= 2|\cent(G)|-2$, we must have $\gamma(\mathcal{C}^{**}(G))= |\cent(G)|-1$ and thus by Proposition \ref{prop:dom-AC-grp}, $G$ has to be an AC-group.
\end{proof}

As the generalized quaternion group $Q_{2^m}$ is AC-group, by the above proposition and \cite[Section 3]{haji2019groups}, 
for $m \geq 3$, $\gamma_t(\mathcal{C}^{**}(Q_{2^m})=2^{m-1}+2 $. We next show the following. 

\begin{proposition}
\label{prop:}
Let $G$ be a group of order $p_1^{\alpha_1} \cdots p_r^{\alpha_r}$ and $Z(G)$ is of order $p_1^{\beta_1} \cdots p_r^{\beta_r}$. If $\sum_{i=1}^r \beta_i = \sum_{i=1}^r \alpha_i -2$, then $G$ is $AC$-group and subsequently,
$\gamma(\mathcal{C}^{**}(G))=|\cent(G)|-1$ and 
$\gamma_t(\mathcal{C}^{**}(G))=2|\cent(G)|-2$.
\end{proposition} 

\begin{proof}
Let $x, y \in G \setminus Z(G)$. As $|C(x)|$ and $|C(y)|$ divides $|G|$, and $|Z(G)|$ divides $|C(x)|$ and $|C(y)|$, using the condition $\sum_{i=1}^r \beta_i = \sum_{i=1}^r \alpha_i -2$, we must have $C(x) \cap C(y)= Z(G)$, completing the proof. 
\end{proof}

\begin{corollary}
Let $G$ be any non-abelian group of order $pq,$ where $p\neq q$ are primes with $p<q.$ Then $\gamma(\mathcal{C}^{**}(G))=q+1.$ Moreover, if $p>2,$ then $\gamma_t(\mathcal{C}^{**}(G))=2(q+1).$    
\end{corollary}

Let $\mathrm{nacent}(G)$ denote the set of all non-abelian element-centralizers of $G$. 

\begin{proposition}
Let $G$ be a non-abelian group which is not generalized dihedral and $|\mathrm{nacent}(G)| = 2$. Then
 \[
\gamma_t(\mathcal{C}^{**}(G)) = 2\big(|\mathrm{cent}(G)| - |\mathrm{cent}(C(a))|\big),
\]
where $C(a)$ is the unique proper non-abelian centralizer of $G$.
\end{proposition}
\begin{proof}
By \cite[Proposition 2.3]{haji2019groups}, we get $\gamma(\mathcal{C}^{**}(G)) = |\mathrm{cent}(G)| - |\mathrm{cent}(C(a))|.$ 
Let 
$D=\{x_1, \cdots, x_t\}$ be a minimal dominating set of $\mathcal{C}^{**}(G)$.
We claim that $D$ is an independent set in $\mathcal{C}^{**}(G).$ Suppose not. Let $x_i\sim x_j,$ for some $i\neq j.$ We consider two centralizers $C(x_i)$ and $C(x_j).$ By the given condition, at least one of these two centralizers is abelian. Suppose $C(x_i)$ is abelian. Since $x_ix_j=x_jx_i$, then it is easy to see that $C(x_i)\subseteq C(x_j).$ Now consider $D'= D \setminus \{x_i \}$. Then $D'$ is also a dominating set, contradicting the minimality of $D$. Hence, $\gamma_t(\mathcal{C}^{**}(G))=2\gamma(\mathcal{C}^{**}(G)).$ This completes the proof. 
\end{proof}

In \cite{haji2019groups}, Haji and Amiri showed that for any extraspecial $p$-group, the domination number of the proper commuting graph is $p+1$. In the following result, we give another family of $p$-groups, for which the domination number is $p+1$. 

\begin{proposition}
\label{cor:heisenberg} 
Let $G$ be any group of size $p^r$. If $|Z(G)|=p^{r-2}$, then \[
\gamma_t(\mathcal{C}^{**}(G))= 2(p+1).
\] In particular, if $G$ is any group of order $p^3,$ we have $\gamma(\mathcal{C}^{**}(G))= p+1$ and $\gamma_t(\mathcal{C}^{**}(G))= 2(p+1).$
\end{proposition}

\begin{proof}
Let $D=\{x_1, \cdots, x_{p+1}\}$ be a minimum dominating set of $\mathcal{C}^{**}(G).$ Clearly, $x_1, \cdots, x_{p+1}\in G\setminus Z(G).$ Note that $|C(x_i)|=p^{r-1},$ and $C(x_i)\cap C(x_j)$ is a subgroup of both $C(x_i)$ and $C(x_j),$ for all $i\neq j\in [p+1].$ Therefore, $C(x_i)\cap C(x_j)=Z(G).$ That is, $D$ is also an independent set in $C^{**}(G).$ Hence $\gamma_t(\mathcal{C}^{**}(G))= 2(p+1).$    
\end{proof}


    


Next, we consider all non-abelian groups of order $p^{4}$, where $p$ is an odd prime. There are $10$ such groups, namely:

\begin{enumerate}
    \item $G_{1} \cong (\mathbb{Z}_{p^{2}} \times \mathbb{Z}_{p}) \rtimes \mathbb{Z}_{p}$.
    
    \item $G_{2} \cong \mathbb{Z}_{p^{2}} \rtimes \mathbb{Z}_{p^{2}}$.
    
    \item $G_{3} \cong \mathbb{Z}_{p^{3}} \rtimes \mathbb{Z}_{p}$.
    
    \item $G_{4} \cong \mathbb{Z}_{p} \times \big((\mathbb{Z}_{p} \times \mathbb{Z}_{p}) \rtimes \mathbb{Z}_{p}\big)$.
    
    \item $G_{5} \cong \mathbb{Z}_{p} \times (\mathbb{Z}_{p^{2}} \rtimes \mathbb{Z}_{p})$.
    
    \item $G_{6} \cong (\mathbb{Z}_{p^{2}} \times \mathbb{Z}_{p}) \rtimes \mathbb{Z}_{p}$.
    
    \item $G_{7} \cong (\mathbb{Z}_{p} \times \mathbb{Z}_{p} \times \mathbb{Z}_{p}) \rtimes \mathbb{Z}_{p}$.
    
    \item $G_{8} \cong (\mathbb{Z}_{p^{2}} \rtimes \mathbb{Z}_{p}) \rtimes \mathbb{Z}_{p}$.
    
    \item $G_{9} \cong (\mathbb{Z}_{p^{2}} \times \mathbb{Z}_{p}) \rtimes \mathbb{Z}_{p}$.
    
    \item $G_{10} \cong (\mathbb{Z}_{p^{2}} \times \mathbb{Z}_{p}) \rtimes \mathbb{Z}_{p}$.
\end{enumerate}

Out of these $10$ groups, we have $|Z(G_i)| = p^{2}$ for all $i \in \{1,\ldots,6\}$, and $|Z(G_i)| = p$ for all $i \in \{7,\ldots,10\}$.

\begin{proposition}
Let $G$ be a finite non-abelian $p$-group of order $p^{4}$, where $p$ is an odd prime. Then
\[
\gamma\big(C^{**}(G)\big)=
\begin{cases}
p+1, & \text{if } G \in \{G_{1},\ldots,G_{6}\}, \\[4pt]
p^{2}+1, & \text{if } G \in \{G_{7},\ldots,G_{10}\}.
\end{cases}
\]
Moreover, 
\[
\gamma_t\big(C^{**}(G)\big)=
\begin{cases}
2(p+1), & \text{if } G \in \{G_{1},\ldots,G_{6}\}, \\[4pt]
2(p^{2}+1), & \text{if } G \in \{G_{7},\ldots,G_{10}\}.
\end{cases}
\]
\end{proposition} 

 \begin{proof}
If $G \in \{G_{1},\ldots,G_{6}\}$, then $|Z(G)| = p^{2}$. 
Therefore, by Theorem~\ref{cor:heisenberg}, we obtain 
that $
\gamma\big(C^{**}(G)\big) = p+1
$ and $
\gamma_t\big(C^{**}(G)\big) = 2(p+1).
$ 
Now, let $G \in \{G_{7},\ldots,G_{10}\}$. 
Then, by \cite[ Corollary~3.1]{MalviyKakkar2025CommutingGraphp4}, it follows that 
$
\gamma\big(C^{**}(G)\big) = p^{2}+1.
$
This completes the proof. 
\end{proof}


\begin{proposition}
Let \( G = \mathrm{Sz}(q) \), where $q=2^{2n+1}.$  Then
$\gamma_t(\mathcal{C}^{**}(G)) = 2\gamma(\mathcal{C}^{**}(G)),$
where
\begin{align*}
\gamma(\mathcal{C}^{**}(G))&=(q^2+1)+\frac{q^2(q^2+1)}{2} + \frac{q^2(q-1)(q^2+1)}{4(q - 2r + 1)} + \frac{q^2(q-1)(q^2+1)}{4(q + 2r + 1)} \text { and}  \\
r&=2^n.
\end{align*}
\end{proposition}
\begin{proof}
The Suzuki group \( G \) contains subgroups \( F, A, B, \) and \( C \), where:
\[
|F| = q^2, \quad |A| = q - 1, \quad |B| = q - 2r + 1, \quad |C| = q + 2r + 1,
\]
(see \cite[Chapter~XI, Theorems 3.10 and 3.11]{Finite-grp-III-Huppert-Blackburn}).

Moreover, \( G \) admits a partition:
\[
\mathcal{P} = \{ A^x, B^x, C^x, F^x : x \in G \},
\]
where for each \( x \in G \), the subgroups \( A^x, B^x, C^x \) are cyclic and \( F^x \) is a Sylow \( 2 \)-subgroup of \( G \). Now, by Proposition~5.3 in~\cite{haji2019groups}, it follows that
\[
\gamma_t\big(\mathcal{C}^{**}(G)\big) = 2\,\gamma\big(\mathcal{C}^{**}(G)\big).
\]
    
\end{proof}

\begin{proposition}
\label{Thm:Dom of Commuting(G),G=pgl(2, q), q odd prime}
Let \( G =\mathrm{PGL}(2,p^n) \), where \( p \) is an odd prime. Then the domination number of $\mathcal{C}^{**}(G)$ is equal to 
\begin{align*}
\gamma(\mathcal{C}^{**}(G))=&p^{2n} + p^n + 1 \text{ and }\\  
\gamma_t(\mathcal{C}^{**}(G))=&2\left(p^{2n} + p^n + 1\right).
\end{align*}

\end{proposition}

\begin{proof}
By \cite[II, Satz 8.5]{Endliche-gruppen-I-Huppert}, the set
\[
\mathcal{P} = \{ U^{x},\ T^{x},\ V^{x} : x \in G \}
\]
is a partition of \( G \). Moreover:
\begin{itemize}
    \item $U$ is an elementary abelian $p$-group of order $p^n$;
    \item $T$ is a cyclic group of order $p^n - 1$;
    \item $V$ is a cyclic group of order $p^n + 1$.
\end{itemize}

The number of conjugates of $U$, $T$, and $V$ in $G$ are
\[
p^n + 1, \quad \frac{p^n(p^n + 1)}{2}, \quad \text{and} \quad \frac{p^n(p^n - 1)}{2},
\]
respectively.

For each $x \in G$, the subgroups $U^{x}$, $T^{x}$, and $V^{x}$ arise as centralizers of certain elements. Moreover, $\mathcal{P}$ forms a partition of the group $G$. Hence,
\begin{align*}
\gamma\big(\mathcal{C}^{**}(G)\big)
&= p^n + 1 + \frac{p^n(p^n + 1)}{2} + \frac{p^n(p^n - 1)}{2} \\
&= p^{2n} + p^n + 1.
\end{align*}

Therefore,
\[
\gamma_t\big(\mathcal{C}^{**}(G)\big) = 2\left(p^{2n} + p^n + 1\right).
\]
\end{proof}

\begin{proposition}
Let $G = \mathrm{PSL}(2,q)$, where $q$ is a power of a prime $p$.

\begin{enumerate}
\item If $q \in \{3,4,5\}$, then
\[
\gamma_t(\mathcal{C}^{**}(G)) = 
\begin{cases}
10 & \text{if } q = 3, \\
42 & \text{if } q = 4 \text{ or } 5.
\end{cases}
\]

\item If $q > 5$, then
\[
\gamma_t(\mathcal{C}^{**}(G)) = 2(q^2 + q + 1).
\]
\end{enumerate}
\end{proposition}
\begin{proof}
When $q=4$ or $q=5$, the groups $\mathrm{PSL}(2,q)$ are AC-groups. Thus, by Proposition \ref{prop:total-dom-AC-group}, we are done. 
When $q>5$,
the proof follows from \cite[Proposition 5.2 ii]{haji2019groups}.   
\end{proof}

\section{Spectrum of $\frac{\gamma(\mathcal{C}^{**}(G))}{|G|}$ } 
\label{sec:spectrum-domination} 

The domination number of the proper commuting graph measures how efficiently the noncentral elements of a group can be controlled through commutation. Since every dominating set is a subset of the vertex set, one trivially has
\[ \gamma(\mathcal{C}^{**}(G)) < |G|.\] 
However, this bound is far from sharp for most groups. It is therefore natural to ask how large the domination number can be relative to the size of the group. Equivalently, one may study the extremal behaviour of the ratio
\[ 
\frac{\gamma(\mathcal{C}^{**}(G))}{|G|}.
\] 
Determining the maximum possible value of this ratio, in a sense, identifies the least commutative groups from the perspective of domination in commuting graphs. The following theorem shows that among all finite non-abelian groups, the symmetric group $S_3$ uniquely maximizes this ratio, and hence provides the extremal example for domination in proper commuting graphs.

\begin{theorem}
\label{thm:domination-number-highest}
For any non-abelian group $G$, 
\[ \frac{\gamma(\mathcal{C}^{**}(G))}{|G|}  \leq \displaystyle \frac{2}{3}\] and equality holds if and only if $G$ is isomorphic to $S_3$. 
\end{theorem}

\begin{proof}
By \cite[Theorem 4.5]{amiri2017groups}, we have 
\[ \mathrm{cent(G)} \geq \frac{2|G|}{3}\]
if and only if $G$ 
is $S_3$, $D_{10}$ or $S_3 \times S_3$. Therefore if $G$ is none of these $3$ groups, we must have 
\[ \frac{\gamma(\mathcal{C}^{**}(G))}{|G|}  < \displaystyle \frac{2}{3}.\]
Moreover, the domination number
of $\mathcal{C}^{**}(D_{10})$ 
is $6$ and by using Proposition \ref{Prop:Com_strong-product} and Theorem \ref{thm:dom-strong-t}, we have 
\[  \frac{\gamma(\mathcal{C}^{**}(S_3 \times S_3))}{|S_3 \times S_3|} = \frac{4}{36}.\] 
Therefore, if $G \neq S_3$, we have 
$ \gamma(\mathcal{C}^{**}(G))  < \displaystyle \frac{2|G|}{3}$,
completing the proof. 
\end{proof}

In the next result, we comment on the spectrum of
the numbers
$\gamma(\mathcal{C}^{**}(G))/|G|$ and
we prove the following.

\begin{theorem}
For any $\frac{1}{2} < r < 1$, there exists a group $G$ that satisfies $\frac{\gamma(\mathcal{C}^{**}(G))}{|G|}=r$ if and only if $r$ is of the form $\frac{k}{2k-1}$ for some positive integer $k\geq2$. 
\end{theorem}

\begin{proof}
We first show that for every 
positive integer $k \geq 2$, there exists a group $G$ whose proper commuting graph satisfies
$\frac{\gamma(\mathcal{C}^{**}(G))}{|G|}=\frac{k}{2k-1}$. 
Consider the dihedral 
group $D_{4k-2}$ of order 
$4k-2$. It is easy to verify 
that the domination number 
of the proper commuting graph of $D_{4k-2}$ is $2k$. 

If $G$ is not generalized dihedral, using Proposition \ref{prop:haji-improved}, we have $\frac{\gamma(\mathcal{C}^{**}(G))}{|G|}< \frac{1}{2}$. 
 Let $D(A)$ be a generalized dihedral group. If $|A|$ is even, by Theorem \ref{thm:dom-gen-dih}, the ratio of the domination number of the proper commuting graph and the size of the group is less than $1/2$. If $|A|$ is odd, say, of the form $2k-1$, then the domination number is $2k$, completing the proof. 
\end{proof}  

\section*{Acknowledgements} 
SB and UJ gratefully acknowledges the support of the NBHM research project (Reference No.~02011/29/2025NBHM(RP)/R\&DII/11951). The authors also thank NBHM, India, for funding this work. Furthermore, the authors appreciate the excellent research environment provided by Dhirubhai Ambani University, Gandhinagar, Gujarat. HKD acknowledges the INSPIRE Faculty Fellowship (Reference No. DST/INSPIRE/\linebreak04/2024/004712; Faculty Registration No. IFA24-MA 205) for support during the preparation of this work, and thanks the Department of Science and Technology (DST), India, for funding. He also appreciates excellent research environment provided by the Department of Mathematics at the Indian Institute of Technology Jammu. SB thanks Peter Cameron for fruitful discussions regarding Theorem \ref{CentralizerIs2-GeneralizedDihedralGrp}. HKD thanks Manideepa Saha for fruitful discussions regarding Theorems \ref{thm:dom-strong-t} and \ref{thm:tot-dom-strong-t}.

\bibliographystyle{amsplain}
\bibliography{gen-inv-lcp.bib}
\end{document}